\documentclass[a4paper,10pt]{article}
\usepackage{amssymb}
\usepackage{amsmath}

\usepackage{tikz}
\usetikzlibrary{matrix,arrows}

\newcommand{\qed}{%
  \ifmmode 
   \eqno{\qedsymbol}
  \else
    \leavevmode\unskip\penalty9999 \hbox{}\nobreak\hfill\hbox{\qedsymbol}
  \fi
}
\newcommand{\qedsymbol}{\leavevmode\vrule height 1.2ex width 1.1ex depth -.1ex}
\newenvironment{proof}{\begin{trivlist}\item[\hskip%
\labelsep{\bf Proof.\quad}]}%
{\hfill\qed\rm\end{trivlist}}

\newtheorem{theorem}{Theorem}[section]
\newtheorem{corollary}[theorem]{Corollary}
\newtheorem{proposition}[theorem]{Proposition}
\newtheorem{lemma}[theorem]{Lemma}

\newtheorem{example}[theorem]{Example}
\newtheorem{remark}[theorem]{Remark}

\mathchardef\emptyset="001F
\def\n{\mathbb N}
\def\im{\hbox{\rm im}}
\def\X{{\mathcal X}}

\def\SF{\mathcal{SF}}
\def\CP{\mathcal{CP}}
\def\Pr{{\mathcal P}}

\def\F{{\mathcal F}}
\def\D{{\mathcal D}}
\def\Fr{{\mathcal Fr}}
\def\PWF{\mathcal{PWF}}
\def\T{\mathcal{T_F}}
\def\WT{\mathcal {WT_F}}
\def\I{\mathcal{I}}

\def\PWI{\mathcal{PWI}}

\begin{document}

\linespread{1}

\title{Covers of acts over monoids II}
\author{Alex Bailey and James Renshaw\\\small Department of Mathematics\\
\small University of Southampton\\
\small Southampton, SO17 1BJ\\
\small England\\
\small Email: alex.bailey@soton.ac.uk \\
\small j.h.renshaw@maths.soton.ac.uk}
\date{June 2012}
\maketitle

\begin{abstract}
In 1981 Edgar Enochs conjectured that every module has a flat cover and finally proved this in 2001. Since then a great deal of effort has been spent on studying different types of covers, for example injective and torsion free covers. In 2008, Mahmoudi and Renshaw initiated the study of flat covers of acts over monoids but their definition of cover was slightly different from that of Enochs. Recently, Bailey and Renshaw produced some preliminary results on the `other' type of cover and it is this work that is extended in this paper. We consider free, divisible, torsion free and injective covers and demonstrate that in some cases the results are quite different from the module case.
\end{abstract}

{\bf Key Words} Semigroups, monoids, acts, torsion free, divisible, injective, covers, precovers

{\bf 2010 AMS Mathematics Subject Classification} 20M50.

\section{Introduction and Preliminaries}

Let $S$ be a monoid. Throughout, unless otherwise stated, all acts will be right $S-$acts and all congruences right $S-$congruences. We refer the reader to~\cite{howie-95} for basic results and terminology in semigroups and monoids and to  \cite{ahsan-08} and \cite{kilp-00} for those concerning acts over monoids. The aim of this paper is to extend the ideas introduced in~\cite{bailey-12} and in particular consider the problem of which monoids $S$ have the property that all their acts have divisible, torsion free or injective covers. These covers have been studied in detail for modules over a unitary ring $R$ but the situation for injective covers in particular is very different in the monoid case. We also consider the CRM monoids introduced by Feller and Gantos in~\cite{feller-69} and generalise the construction of the semigroup of quotients to acts. We provide an application of this to consider covers of {\em weakly torsion free} acts.

After some preliminary results and definitions in Section 1 and a brief section on free covers in Section 2, we give a necessary and sufficient condition for acts to have covers with the unique mapping property in Section 3.
Divisible covers are considered in Section 4 where we show that if an act has a divisible cover then it is precisely its largest divisible subact and show that not every act has a divisible cover. In the early sixties, Enochs proved that all modules over integral domains have torsion free covers and we provide a similar result in Section 5 using right cancellative monoids. In~\cite{feller-69}, Feller and Gantos gave a weaker condition than torsion free and provided connections with what they referred to as CRM monoids. We consider these concepts in Section 6 where we show how to construct the classical semigroup of right quotients using a more general construction for acts. Enochs proved that all $R-$modules have an injective cover if and only if $R$ is Noetherian. The situation for monoids is quite different. If $R$ is a ring and $f:X\to Y$ is a split $R-$monomorphism then it is well known that $\im(f)$ is a direct summand of $Y$. From this we can deduce that an indecomposable injective $R-$module is the injective hull of all of its submodules. This is not so in the category of $S-$acts. We do however provide some necessary conditions for the existence of injective covers of $S-$acts. For principally weakly injective covers the situation is slightly easier and we provide a sufficient condition concerning these covers in Section 8.

Let $S$ be a monoid, and $A$ be an $S-$act. Unless otherwise stated, in the rest of this section, $\X$ will be a class of $S-$acts closed under isomorphisms. By an $\X$-{\em precover} of $A$ we mean an $S-$map $g: P\to A$ for some $P\in \X$ such that
for every $S-$map $g':P'\to A$, for $P'\in \X$, there exists an $S-$map $f:P'\to P$ with $g'=gf$.
$$
\begin{tikzpicture}[description/.style={fill=white,inner sep=2pt}]
\matrix (m) [matrix of math nodes, row sep=3em,
column sep=2.5em, text height=1.5ex, text depth=0.25ex]
{P & A\\ 
&P'\\};
\path[->,font=\scriptsize]
(m-2-2) edge node[auto,below left] {$f$} (m-1-1)
(m-2-2) edge node[auto, right] {$g'$} (m-1-2)
(m-1-1) edge[->] node[auto,above] {$g$} (m-1-2);
\end{tikzpicture}
$$

If in addition the precover satisfies the condition that each $S-$map $f:P\to P$ with $gf=g$ is an isomorphism, then we shall call it an {\em $\X-$cover}. We shall frequently identify the (pre)cover with its domain. We refer the reader to~\cite{bailey-12} for more detailed information surrounding covers and precovers and the connection with colimits and directed colimits. We include however a few necessary results from that paper here for completeness.

\begin{theorem}[{\cite[Theorem 4.11]{bailey-12}}]\label{precover-cover-theorem}
Let $S$ be a monoid, let $A$ be an $S-$act and let $\X$ be a class of $S-$acts closed under directed colimits. If $A$ has an $\X-$precover then $A$ has an $\X-$cover.
\end{theorem}

Let $S$ be a monoid and let $\X$ be a class of $S-$acts. We say that $\X$ satisfies the {\em (weak) solution set condition} if for all $S-$acts $A$ there exists a set $S_A\subseteq\X$ such that for all (indecomposable) $X\in\X$ and all $S-$maps $h:X\to A$ there exists $Y\in S_A$, $f:X\to Y$ and $g:Y\to A$ such that $h=gf$.

\begin{theorem}[{\cite[Theorem 4.13]{bailey-12}}]\label{solution-set-theorem}
Let $S$ be a monoid and let $\X$ be a class of $S-$acts such that
$X_i \in \X$ for each $i \in I\Rightarrow\dot\bigcup_{i \in I}X_i \in \X$.
Then every $S-$act has an $\X-$precover if and only if
\begin{enumerate}
\item for every $S-$act $A$ there exists an $X$ in $\X$ such that $\text{Hom}_S(X,A)\ne\emptyset$;
\item $\X$ satisfies the solution set condition;
\end{enumerate}
\end{theorem}

\begin{proposition}[{\cite[Corollary 4.14]{bailey-12}}]\label{lambda-skeleton-proposition}
Let $S$ be a monoid and let $\X$ be a class of $S-$acts such that
\begin{enumerate}
\item $\dot\bigcup_{i \in I}X_i \in \X \Leftrightarrow X_i \in \X$ for each $i \in I$;
\item for every $S-$act $A$ there exists an $X$ in $\X$ such that $\text{Hom}_S(X,A)\ne\emptyset$;
\item there exists a cardinal $\lambda$ such that for every indecomposable $X$ in $\X$, $|X|<\lambda$.
\end{enumerate}
Then every $S-$act has an $\X-$precover.
\end{proposition}

\medskip

Recall that an $S-$act $F$ is called {\em free} (with basis $A$) if there exists a set $A$ such that $X=A\times S$ with multiplication given by $(a,s)t=(a,st)$.
An $S-$act $P$ is called {\em projective} if given any $S-$epimorphism $f:A \rightarrow B$, whenever there is an $S-$map $g:P \rightarrow B$ there exists an $S-$map $h:P \rightarrow A$ such that $hf=g$.
An $S-$act $A$ is called {\em torsion free} if for any $x,y \in A$, and for any right cancellative element $c \in S$, $xc=yc$ implies $x=y$.
An $S-$act $A$ is called {\em divisible} if for any $a \in A$, and for any left cancellative $c \in S$, there exists $d \in A$ such that $a=dc$.
An $S$-act $Q$ is {\em injective} if for any monomorphism $\iota:A \to B$ and any homomorphism $f:A \rightarrow Q$ there exists a homomorphism $\bar{f}:B \rightarrow Q$ such that $f=\bar{f}\iota$ and it is {\em principally weakly injective} if it is injective with respect to all inclusion of principal ideals of $S$.

\smallskip

Throughout this paper we shall denote the class of all free $S$-acts by $\Fr$, the class of all projective $S-$acts by $\Pr$, the class of all torsion free $S-$acts by $\T$, the class of all divisible $S-$acts by $\D$, the class of all injective $S-$acts by $\I$ and the class of all principally weakly injective $S-$acts by $\PWI$. It is well known that $\Fr\subseteq\Pr\subseteq\T$ and $\I\subseteq\D$.

\section{Free covers}

It is well known that not every $S-$act has a $\Pr-$cover. In fact monoids over which every right $S-$act has a $\Pr-$cover are called {\em right perfect monoids} (see~\cite{isbell-71} and~\cite{fountain-76}). It was shown however in \cite[Proposition 5.8]{bailey-12} that every right $S$-act has a $\Pr-$precover. We prove similar results here for $\Fr-$(pre)covers.

\smallskip

Let $S$ be a monoid and $f : C \to A$ be an S-epimorphism. We
call $f$ {\em coessential} if for each $S-$act $B$ and each $S-$map $g : B \to C$, if $fg$ is an epimorphism then $g$ is an epimorphism. It is fairly clear that $f:C\to A$ is a coessential epimorphism if and only if there is no proper $S-$subact $B$ of $C$ such that $f|_B$ is onto.

\begin{lemma}
Let $S$ be a monoid and let $A$ be an $S-$act. Then $f:C\to A$ is a $\Fr-$cover of $A$ if and only if $f$ is a coessential epimorphism with $C\in\Fr$.
\end{lemma}

\begin{proof}
Suppose that $g$ is a $\Fr-$cover of $S$. Then by ~\cite[Proposition 4.3]{bailey-12} $g$ is an epimorphism. Let $B$ be a subact of $C$
such that $g|_B$ is an epimorphism. Then since $C$ is projective, there exists an $S-$map $h : C \to B$ with $(g|_B)h = g$. Then we get easily that $g = g\iota h$, where $\iota: B\to C$ is the inclusion map. Now, by hypothesis, $\iota h$ must be an isomorphism which gives $B = C$.

Conversely let $g : C\to A$ be a coessential epimorphism and suppose that $C\in\Fr$. Then $g$ is a $\Fr-$precover since every free $S-$act is projective. To prove that it is a $\Fr-$cover, let $f : C\to C$ be an $S-$map
with $g = gf$. Then, $g|_{\im(f)}$ is onto, and so $\im(f) = C$. Thus $f$ is an epimorphism, and since $C$ is projective, there exists an $S-$map $h : C\to C$ such that $fh = 1_C$. So $h$ is a monomorphism and $gh = (gf)h = g(fh) = g$. Thus, $g|_{\im(h)}$ is onto,
and hence $\im(h) = C$. Therefore, $h$ is an epimorphism and so an isomorphism. Hence $f$ is an isomorphism.
\end{proof}

\begin{lemma}\label{free-precover-lemma}
Let $S$ be a monoid. Then every $S-$act has a $\Fr-$precover.
\end{lemma}

\begin{proof}
Let $A$ be an $S-$act. Take $A \times S$ the free $S-$act generated by $A$ with the $S-$map $g:A \times S \rightarrow A$, $(a,s) \mapsto as$. Then $g$ is an $S-$epimorphism and so every free $S-$act (which is also projective) factors through it. 
\end{proof}

Notice that these $\Fr-$precovers are also $\Pr-$precovers.

\begin{theorem}
Given any monoid $S$, the following are equivalent:
\begin{enumerate}
\item Every $S-$act has an $\Fr$-cover.
\item The one element $S-$act $\Theta$ has an $\Fr$-cover.
\item $S$ is a group.
\end{enumerate}
\end{theorem}

\begin{proof}
$(1) \Rightarrow (2)$ is a tautology. \\
$(2) \Rightarrow (3)$ If $g:C\to\Theta$ is a $\Fr-$cover for $\Theta$ then $C=A\times S$ for some set $A$. Let $a\in A$ and define $f:C\to C$ by $f(x,s)=(a,s)$ for $x\in A$. Then $gf=g$ and so $f$ is an automorphism. Hence $|A|=1$ and so $C\cong S$. Now let $x\in S$ and consider $h:S\to S$ given by $h(s)=xs$. Then $gh=g$ and so $h$ is an automorphism and hence $S=xS$ for all $x\in S$. Hence $S$ is a group. \\
$(3) \Rightarrow (1)$ By the homological classification of monoids, $S$ is a group if and only if every strong flat $S-$act is free \cite[Theorem 2.6]{knauer-81}. In particular, since the strongly flat $S-$acts are closed under directed colimits \cite[Proposition 5.2]{stenstrom-71}, the free $S-$acts are also closed under directed colimits and the result follows from Lemma~\ref{free-precover-lemma} and Theorem~\ref{precover-cover-theorem}.
\end{proof}

\section{Covers with the unique mapping property}

An $\X-$(pre)cover $g:X \to A$ of an $S-$act $A$ is said to have the {\em unique mapping property} if whenever there is an $S-$map $h:X' \to A$ with $X' \in \X$, there is a unique $S-$map $f:X' \to X$ such that $h=gf$.

Clearly an $\X-$precover with the unique mapping property is an $\X-$cover with the unique mapping property as the unique identity map is an isomorphism.

\medskip

Note that every act having an $\X-$cover with the unique mapping property is equivalent to saying that $\X$ is a {\em coreflective subcategory} of the category of all $S-$acts. That is to say, the inclusion functor has a right adjoint. See \cite[Exercises 3.J and 3.M]{freyd-64} for more details and from which some of the next results are based.

\begin{lemma} \label{UMP-by-colimits}
Let $S$ be a monoid and let $\X$ be a class of $S-$acts closed under colimits. If an $S-$act has an $\X-$precover then it has an $\X-$cover with the unique mapping property.
\end{lemma}

\begin{proof}
If an $S-$act $A$ has an $\X-$precover, then by Theorem~\ref{precover-cover-theorem} it has an $\X-$cover, say $g:C \to A$. Let $f_1,f_2$ be two endomorphisms of $C$ such that $gf_1=gf_2=g$, we intend to show that $f_1=f_2$ and so the unique mapping property holds. Let $(h,E)$ be the coequalizer of $f_1$ and $f_2$ in $C$, so that by \cite[Proposition II.2.21]{kilp-00}, $E=C/\rho$ where $\rho$ is the smallest congruence generated by the pairs $\{(f_1(c),f_2(c)) : c \in C\}$. Since $g(f_1(c))=g(c)=g(f_2(c))$ it is clear that $\rho \subseteq \ker(g)$. Since $\X$ is closed under colimits $E \in \X$ and by \cite[Proposition 4.16]{bailey-12}, $\rho=id_C$ and hence $f_1=f_2$.
\end{proof}

\begin{lemma}
Let $S$ be a monoid and let $\X$ be a class of $S-$acts. If every $S-$act has an $\X-$cover with the unique mapping property then $\X$ is closed under colimits.
\end{lemma}

\begin{proof}
Let $(X_i,\phi_{i,j})_{i \in I}$ be a direct system of $S-$acts with colimit $(X,\alpha_i)$. Let $g:C \to X$ be the $\X$-cover of $X$ so that for each $i \in I$ there exists a unique $f_i:X_i \to C$ with $gf_i=\alpha_i$. Note that if $i\le j$ then $gf_i=\alpha_i=\alpha_j \phi_{i,j}=(gf_j)\phi_{i,j}=g(f_j\phi_{i,j})$ and so by the unique mapping property $f_i=f_j \phi_{i,j}$ for all $i \le j$. Hence by the colimit property, there exists a unique $S-$map $f:X \to C$ such that $f \alpha_i=f_i$ for all $i \in I$. Therefore $\alpha_i=gf_i=g(f \alpha_i)=(gf)\alpha_i$ and since, by the colimit property, there exists a unique $S-$map $h:X \to X$ with $h \alpha_i=\alpha_i$ for all $i \in I$, we clearly have $gf=id_X$. But then $g(fg)=(gf)g=g$ and by the unique mapping property $fg=id_C$ and so $X$ is isomorphic to $C \in \X$.
\end{proof}

Hence by Theorem~\ref{solution-set-theorem} we have the following

\begin{theorem}
Let $S$ be a monoid and $\X$ a class of $S-$acts. Every $S-$act has an $\X-$cover with the unique mapping property if and only if
\begin{enumerate}
\item $\X$ is closed under colimits.
\item For every $S-$act $A$ there exists $X \in \X$ such that Hom$(X,A) \neq \emptyset$.
\item $\X$ satisfies the solution set condition.
\end{enumerate}
\end{theorem}

Recall from \cite[Theorem II.3.16]{kilp-00} that an $S-$act $G$ is called a {\em generator} if there exists an epimorphism $G \to S$.

\begin{theorem} \label{generator-and-colimits}
Let $S$ be a monoid and let $\X$ be a class of $S-$acts containing a generator which is closed under colimits. Then every $S-$act has an $\X-$cover with the unique mapping property. 
\end{theorem}

\begin{proof}
Let $G \in \X$ be a generator with $S-$epimorphism $h:G \to S$. Given any $S-$act $A$, let $A \times G$ be the $S-$act with the action on the right component, so that we have an $S-$epimorphism $g_A:A \times G \to A$, $(a,y) \mapsto ah(y)$. Notice that $A\times G$ is isomorphic to a coproduct of $|A|$ copies of $G$ and so $A\times G\in\X$. Consider, up to isomorphism, the set $(X_i,g_i,f_i)_{i \in I}$ of all $S-$acts $X_i \in \X$ and $S-$epimorphisms $g_i:A \times G \to X_i$ such that there exist $f_i:X_i \to A$ with $f_i g_i=g_A$. Notice that $(A\times G, 1_{A\times G},g_A)$ is one such triple and so $I\ne\emptyset$, and that this is indeed a set since $|X_i| \le |A \times G|$. Define an order on this set $(X_i,g_i,f_i) \leq (X_j,g_j,f_j)$ if and only if there exists $\phi_{i,j}:X_i \to X_j$ with $\phi_{i,j}g_i=g_j$ and $f_j\phi_{i,j}=f_i$.
$$
\begin{tikzpicture}[description/.style={fill=white,inner sep=2pt}]
\matrix (m) [matrix of math nodes, row sep=2em,
column sep=2em, text height=1.5ex, text depth=0.25ex]
{& A\times G&\\ 
X_i&&X_j\\
&A&\\};
\path[->,font=\scriptsize]
(m-1-2) edge node[auto,left] {$g_i$} (m-2-1)
(m-1-2) edge node[auto, right] {$g_j$} (m-2-3)
(m-2-1) edge node[auto,above] {$\phi_{i,j}$} (m-2-3)
(m-2-1) edge node[auto,left] {$f_i$} (m-3-2)
(m-2-3) edge node[auto,right] {$f_j$} (m-3-2);
\end{tikzpicture}
$$
Notice that since $g_i$ is onto then if such a $\phi_{i,j}$ exists then it is unique. It is not hard to see that this is a partial order, and $(X_i,\phi_{i,j})_{i \in I}$ is a direct system. In fact, this order has a least element $(X_0,1_{A \times G},g_A)$, where $X_0=A \times G$ and $\phi_{0,i}=g_i$ for all $i \in I$. Let $(M,\alpha_i)$ be the colimit of this system, since each $\phi_{i,j}$ is an epimorphism, so are the $\alpha_i$ and since $\X$ is closed under colimits, $M \in \X$. Since $f_j\phi_{i,j}=f_i$ for all $i \leq j \in I$ there must exist some $f:M \to A$ such that $f\alpha_i=f_i$ for all $i \in I$. Since $M \in \X$ and $\alpha_0$ is an epimorphism we see that $(M,\alpha_0,f)$ is in fact a maximal element in the ordering.

We claim that $f:M \to A$ is an $\X$-precover of $A$. Given any $X \in \X$ with $S-$map $\sigma:X \to A$, let $g_X:X \times G \to X$, $(x,y) \mapsto xh(y)$ be an $S-$epimorphism.
As before, observe that $A\times G, X\times G\in\X$.
Define $m:X \times G \to A \times G$ by $m(x,y) = (\sigma(x),y)$ and consider the pushout diagram
$$
\begin{tikzpicture}[description/.style={fill=white,inner sep=2pt}]
\matrix (m) [matrix of math nodes, row sep=3em,
column sep=2.5em, text height=1.5ex, text depth=0.25ex]
{X\times G & A\times G\\ 
X&Q\\};
\path[->,font=\scriptsize]
(m-2-1) edge node[auto,below] {$q_2$} (m-2-2)
(m-1-2) edge node[auto, right] {$q_1$} (m-2-2)
(m-1-1) edge node[auto,above] {$m$} (m-1-2)
(m-1-1) edge node[auto,left] {$g_X$} (m-2-1);
\end{tikzpicture}
$$
Since $g_X$ is an epimorphism then so is $q_1$ \cite[Lemma I.3.6]{renshaw-85} and since $\X$ is closed under colimits then $Q \in \X$. By \cite[Proposition II.2.16]{kilp-00}, $Q=(X \dot\bigcup (A \times G))/\rho$ where $\rho=\{(m(z),g_X(z)) : z \in X \times G\}^\#$. Since $g_Am=\sigma g_X$ then there exists a unique $\psi:Q\to A$ such that $\psi q_1=g_A, \psi q_2 = \sigma$
and so by the maximality of $(M,\alpha_0,f)$ there exists an $S-$map $\phi:Q \to M$ such that $\phi q_1=\alpha_0$. It is straightforward to check that $f \phi q_2 =\sigma$, and so $f:M \to A$ is an $\X-$precover of $A$. Since $\X$ is closed under colimits, we can apply Lemma \ref{UMP-by-colimits}.
\end{proof}

So by \cite[Corollary 4.4]{bailey-12} we get the following result

\begin{corollary}
Let $S$ be a monoid and let $\X$ be a class of $S-$acts. Every $S$-act has an epimorphic $\X-$cover with the unique mapping property if and only if $\X$ contains a generator and is closed under colimits.
\end{corollary}

\medskip

Note that although $\SF$ (strongly flat acts), $\CP$ (condition $(P)$ acts), $\F$ (flat acts) and $\T$ all contain $S$ as a generator, they are only closed under directed colimits not all colimits in general. However, as we shall see in the next section, the class $\D$ of all divisible $S-$acts is closed under colimits.

\section{Divisible covers}

As mentioned previously, an obvious necessary condition for an $S-$act $A$ to have an $\X-$cover is the existence of an $S-$act $C\in X$ such that $\hbox{\rm Hom}_S(C,A)\ne\emptyset$. It is fairly obvious that if $\X$ includes all the free acts then this condition is always satisfied. We consider here the class of divisible acts where this condition is not always satisfied and where the covers, when they exist, are monic rather than epic.

\begin{proposition}[{\cite[Proposition III.2.4]{kilp-00}}] \label{divisible-results}
Let $S$ be a monoid
\begin{enumerate}
\item Any homomorphic image of a divisible $S-$act is divisible.
\item $\dot\bigcup_{i \in I}D_i$ is divisible if and only if each $D_i$ are divisible.
\end{enumerate}
\end{proposition}

\begin{lemma}\label{d-closed-lemma}
$\D$ is closed under colimits.
\end{lemma}

\begin{proof}
Let $(X_i,\phi_{i,j})_{i \in I}$ be a direct system of divisible $S-$acts and let $(X,\alpha_i)$ be the colimit. For each $x \in X$ and left cancellative $c \in S$ there exists $x_i \ \in X_i$ with $\alpha_i(x_i)=x$ and, since $X_i$ is divisble, there exists $d_i \in X_i$ such that $x_i=d_ic$. So $x=\alpha_i(x_i)=\alpha_i(d_i c)=\alpha_i(d_i)c$ and $X$ is divisible.
\end{proof}

However, although $\D$ is closed under colimits, it does not always contain a generator. In fact we have the following

\begin{lemma}
Let $S$ be a monoid, then the following are equivalent
\begin{enumerate}
\item $\D$ has a generator.
\item $S$ is divisible
\item All left cancellative elements of $S$ are left invertible.
\item Every $S-$act is divisible.
\item Every $S-$act has an epimorphic $\D-$cover
\end{enumerate}
\end{lemma}

\begin{proof}
The equivalence of $(2), (3)$ and $(4)$ follows by \cite[Proposition III.2.2]{kilp-00}.

$(1) \Rightarrow (2)$ If $G \in \X$ is a generator, then there exists an epimorphism $g:G \to S$. Hence $S$ is the homomorphic image of a divisible $S-$act and so is divisble. \\
$(4) \Rightarrow (5)$ Every $S-$act is its own epimorphic $\D-$cover. \\
$(5) \Rightarrow (1)$ The epimorphic $\D-$cover of $S$ is a generator in $\D$.
\end{proof}

Since the union of a set of divisible acts is itself divisible then every $S-$act $A$ which contains a divisible subact has a unique largest divisible subact $D_A=\bigcup_{i \in I}D_i$ where $\{D_i : i \in I\}$ is the set of all divisible subacts of $A$.

\begin{theorem}
Let $S$ be a monoid and $A$ an $S-$act. Then we have the following
\begin{enumerate}
\item $D$ is a $\D-$precover of $A$.
\item $D$ is a $\D-$cover of $A$.
\item $D$ is a $\D-$cover of $A$ with the unique mapping property.
\item $D=D_A$ is the largest divisible subact of $A$.
\end{enumerate}
\end{theorem}

\begin{proof}
Clearly $(3) \Rightarrow (2) \Rightarrow (1)$ and $(1) \Rightarrow (3)$ by Lemmas \ref{UMP-by-colimits} and \ref{d-closed-lemma}. \\
$(4) \Rightarrow (1)$ Let $X$ be a divisible $S-$act $h:X \to A$ a homomorphism. By Proposition \ref{divisible-results}, $\im(h)$ is a divisible subact of $A$ and so $\im(h) \subseteq D$. Therefore $h:X \to D$ is a well-defined $S-$map obviously commuting with the inclusion map. Hence $D$ is a $\D-$precover of $A$. \\
$(2) \Rightarrow (4)$ Let $g:D \to A$ be a $\D-$cover of $A$. Then as before, $A$ contains a divisible subact and so by the previous case, the largest divisible subact of $A$ is also a $\D-$cover of $A$. But $\D-$covers are unique up to isomorphism.
\end{proof}

We therefore have the following result

\begin{theorem}
Let $S$ be a monoid. Then the following are equivalent
\begin{enumerate}
\item Every $S-$act has a $\D-$cover.
\item Every $S-$act has a divisible subact.
\item $S$ contains a divisible right ideal $K$.
\end{enumerate}
\end{theorem}

\begin{proof}
The equivalence of $(1)$ and $(2)$ are obvious by the last theorem. \\
If every $S-$act has a divisible subact then clearly $S$ has a divisible subact, which is a right ideal. Conversely if $K$ is a divisible subact of $S$, then given any $S-$act $X$, it has a divisible subact $XK$. Hence (2) and (3) are equivalent.
\end{proof}

For example, if $S$ is any monoid with a left zero, $z$ then $K=\{z\}$ is a divisible right ideal of $S$ and so every $S-$act has a $\D-$cover.

\smallskip

Notice that not every $S-$act has a $\D-$cover. For example, let $S=(\n,+)$ and consider $S$ as an $S-$act over itself. For every $n \in S$, $n+1$ is a left cancellative element in $S$, but there does not exist $m \in S$ such that $n=m+(n+1)$.Therefore $S$ does not have have any divisible right ideals.

\section{Torsion free covers}

In 1963 Enochs proved that over an integral domain, every module has a torsion free cover \cite{enochs-63}.
We give a proof of the semigroup analogue of Enochs' result that over a right cancellative monoid, every right act has a torsion free cover.

If $A \in \T$, then clearly $B \in \T$ for every subact $B \subseteq A$.

\begin{remark} \label{tf-homsets}
Clearly $S \in \T$ and so for every $S-$act $A$, there exists $X \in \T$ such that Hom$(X,A)\neq \emptyset$.
\end{remark}

\begin{lemma} \label{tf-coproducts}
$\dot\bigcup_{i \in I}A_i \in \T$ if and only if $A_i \in \T$ for each $i \in I$.
\end{lemma}

\begin{proof}
Let $A=\dot\bigcup_{i \in I}A_i$ and suppose $A_i$, $i \in I$ are torsion free $S-$acts. Let $xc=yc$ for some $x,y \in A$, where $c$ is a right cancellative element of $S$. The equality $xc=yc$ implies $x$ and $y$ are in the same connected component, so there exists some $i \in I$ such that $x,y \in A_i$. Since $A_i$ is torsion free, $x=y$ and $A$ is torsion free. Conversely each $A_i$ is a subact of $A$ and so if $A$ is torsion free, each $A_i$, $i \in I$ are torsion free.
\end{proof}

\begin{lemma} \label{tf-direct-limit}
$\T$ is closed under directed colimits.
\end{lemma}

\begin{proof}
Let $(A_i,\phi_{i,j})$ be a direct system of torsion free $S-$acts over a directed index set $I$ with directed colimit $(A,\alpha_i)$. Assume $xc=yc$ where $c$ is a right cancellative element in $S$ and $x,y \in A$. Then there exists $x_i \in A_i$ and $y_j \in A_j$ with $x=\alpha_i(x_i)$, $y=\alpha_j(y_j)$. So $\alpha_i(x_i)c=\alpha_i(x_ic)=\alpha_j(y_jc)=\alpha_j(y_j)c$ and since $I$ is directed, by \cite[Lemma 2.1]{bailey-12} there exists some $k \geq i,j$ such that $\phi_{i,k}(x_i)c=\phi_{i,k}(x_ic)=\phi_{j,k}(y_jc)=\phi_{j,k}(y_j)c$. Since $A_k$ is torsion free $\phi_{i,k}(x_i)=\phi_{j,k}(y_j)$ and $x=\alpha_k\phi_{i,k}(x_i)=\alpha_k\phi_{j,k}(y_j)=y$ as required.
\end{proof}

\begin{theorem}
Let $S$ be a right cancellative monoid, then every $S-$act has a $\T$-cover.
\end{theorem}

\begin{proof}
Let $A$ be an indecomposable torsion free $S-$act. For each $xs=x's \in A$, $s \in S$, since $s$ is right cancellative, $x=x'$. Hence for each $x \in A$, $s \in S$ there is no more than one solution to $x=ys$. Now let $x,y \in A$ be any two elements. Since $A$ is indecomposable there exist $x_1,\ldots,x_n \in A$, $s_1,\ldots,s_n,t_1,\ldots, t_n \in S$ such that $x=x_1s_1$, $x_1t_1=x_2s_2$, \ldots, $x_nt_n=y$, as shown below.

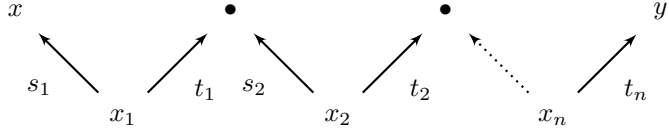
\begin{center}
\begin{tikzpicture}[auto, outer sep=3pt, node distance=2cm,>=latex']
\node (x) {$x$};
\node[below right of = x] (x1) {$x_1$};
\node[above right of = x1] (y1) {$\bullet$};
\node[below right of = y1] (x2) {$x_2$};
\node[above right of = x2] (y2) {$\bullet$};
\node[below right of = y2] (xn) {$x_n$};
\node[above right of = xn] (y) {$y$};
\draw[->,thick] (x1) --  node {$s_1$} (x);
\draw[->,thick] (x1) --  node [swap] {$t_1$} (y1);
\draw[->,thick] (x2) --  node {$s_2$} (y1);
\draw[->,thick] (x2) --  node [swap] {$t_2$} (y2);
\draw[->,dotted,thick] (xn) -- node {} (y2);
\draw[->,thick] (xn) --  node [swap] {$t_n$} (y);
\end{tikzpicture}
\end{center}

If we can show there is a bound on the number of such paths, then there is a bound on the number of elements in $A$. Now, by the previous argument, there are only $|S|$ possible $x_1 \in A$ such that $x=x_1s_1$ for some $s_1 \in S$. In a similar manner, given $x_1$ there are only $|S|$ possible $x_1t_1$ for some $t_1 \in S$. Continuing in this fashion we see that the number of such paths of length $n \in \n$ is bounded by $|S|^{2n}$, and so $|A| \leq |S|^{\aleph_0}$. So by Remark \ref{tf-homsets}, Lemma \ref{tf-coproducts} and Proposition~\ref{lambda-skeleton-proposition} every $S-$act has a $\T-$precover. By Lemma \ref{tf-direct-limit} and Theorem~\ref{precover-cover-theorem} every $S-$act has a $\T-$cover.
\end{proof}

It was shown in \cite[Corollary 2.2]{bulman-fleming-90a} that over a right cancellative monoid, an act is torsion free if and only if it is principally weakly flat, so we get the following corollary.

\begin{corollary}
Every act over a right cancellative monoid has a $\PWF$-cover.
\end{corollary}

\section{Weakly Torsion Free Acts and CRM monoids}\label{weak-torsion-section}

An $S-$act $A$ is called {\em weakly torsion free} if for any $x,y \in A$, and for any cancellative element $c \in S$, $xc=yc$ implies $x=y$. This is the definition of torsion free given in \cite{feller-69} and it is clear that every torsion free right $S-$act is weakly torsion free. We shall denote the class of weakly torsion free right $S-$acts by $\WT$. We are motivated in this section by some of the results presented in~\cite{enochs-63}.

\begin{lemma}\label{t-precover-subact-lemma}
Suppose that $A$ is a right $S-$act and that $\psi:C\to A$ is a $\WT-$precover. If $B\subseteq A$ is an $S-$subact of $A$ and if $D=\psi^{-1}(B)$ then $\psi|_D:D\to B$ is a $\WT-$precover of $B$.
\end{lemma}
\begin{proof} 
Since $S-$subacts of weakly torsion free acts are weakly torsion free then $D$ is weakly torsion free. Suppose then that $X\in\WT$ and $f:X\to B$ is an $S-$map. Then clearly there exists $g:X\to C$ with $\psi g=f$. It is also clear that $\im(g)\subseteq D$ and the result follows.
\end{proof}

It is clear that the previous lemma is also true if we replace $\WT$ by $\T$.

Let $X$ be an $S-$act and consider the {\em weak torsion relation}
$$
\sigma_X=\{(x,y)\in X\times X| xc=yc\text{ for some cancellative }c\in S\}.
$$

\medskip

In~\cite{feller-69} a monoid is said to satisfy the {\em Common Right Multiple Condition} (CRM) if for all $s,c\in S$ with $c$ cancellative there exists $t,d\in S$ with $d$ cancellative such that $sd=ct$. For example, every monoid in which the cancellative elements form a group, every commutative monoid and every left reversible cancellative monoid is a CRM monoid.
If $S$ is a CRM monoid and $X$ an $S-$act then $\sigma_X$ is a congruence on $X$. To see this note that $\sigma_X$ is obviously reflexive and symmetric. Suppose then that $(x,y), (y,z)\in \sigma_X$. Then $xc_1=yc_1$ and $yc_2=zc_2$ for some cancellative $c_1, c_2 \in S$. Then there exists $t,d\in S$ with $d$ cancellative such that $c_1d=c_2t$ and so
$$
xc_1d = yc_1d = yc_2t = zc_2t = zc_1d
$$
and since $c_1d$ is cancellative then $(x,z)\in\sigma_X$ and $\sigma_X$ is an equivalence. Suppose now that $xc=yc$ with $c$ cancellative and suppose also that $s\in S$. Then there exists $d,t\in S$ with $d$ cancellative such that $sd=ct$. Hence $(xs)d = xct = yct = (ys)d$ and so $(xs,ys)\in\sigma_X$ and $\sigma_X$ is a congruence.

\begin{lemma}\label{injective-envelope-torsion-lemma}
Let $S$ be a CRM monoid. If $X$ is a weakly torsion free right $S-$act and if $E(X)$ is the injective envelope of $X$ then $E(X)$ is also weakly torsion free.
\end{lemma}
\begin{proof}
Consider the composite $X\to E(X)\to E(X)/\sigma_{E(X)}$. Since both $X$ and $E(X)/\sigma_{E(X)}$ are weakly torsion free then this composite is a monomorphism. Hence so is the map $E(X)\to E(X)/\sigma_{E(X)}$ and so $E(X)\cong E(X)/\sigma_{E(X)}$ and $E(X)$ is weakly torsion free.
\end{proof}

\begin{lemma}\label{torsion-free-injective-lemma}
Let $S$ be a CRM monoid and let $\X=\WT\cap\D$. If an injective $S-$act has an $\X-$precover then it has a $\WT-$precover.
\end{lemma}
\begin{proof}
Let $E$ be an injective $S-$act and let $g:C\to E$ be an $\X-$precover. Given any $X \in \WT$ and $h:X \to E$ there exists $\phi:E(X) \to E$ such that $\phi|_X=h$. Since $C$ is an $\X-$precover and $E(X) \in \X$ there exists $f:E(X) \to C$ such that $gf=\phi$. Hence $gf|_X=h$ and $C$ is a $\WT-$precover of $E$.
\end{proof}

Let $S$ be a CRM semigroup and let $C$ be the submonoid of cancellative elements of $S$. Let $X$ be a right $S-$act and on $X\times C$ define a relation by
$$
\rho_X=\{((x,c),(y,d)) | \exists s,t\in S, xs=yt, cs=dt\in C\}.
$$
Notice that since $S$ satisfies the CRM property then there exists $s_1\in S, c_1\in C$ with $(dt)s_1=dc_1$ and so $ts_1=c_1$. Consequently $x(ss_1) = yc_1$ and $c(ss_1) = dc_1$. Hence we see that
$$
\rho_X=\{((x,c),(y,d)) | \exists s\in S, t\in C, xs=yt, cs=dt\}.
$$
Then it is an easy matter to show that $\rho_X$ is an equivalence on $X\times C$. Let $Q_X=(X\times C)/\rho_X$ and let $\rho=\rho_S$ be the respective equivalence on $S\times C$. Let $Q=Q_S$ and define a map $Q_X\times Q\to Q_X$ as follows.
$$
\left((x,c)\rho_X\right)\left((s,d)\rho\right)=(xs_1,dc_1)\rho_X\text{ where }cs_1=sc_1\text{ for }s_1\in S, c_1\in C.
$$
First notice that such an $s_1,c_1$ exist by the CRM property. Second, suppose that $cs_2=sc_2$ for $s_2\in S, c_2\in C$. By the CRM property, there exists $c_3\in C, s_3\in S$ such that $c_2c_3=c_1s_3$. Hence
$$
cs_2c_3=sc_2c_3=sc_1s_3=cs_1s_3
$$
and so $s_2c_3=s_1s_3$ as $c$ is cancellative. Therefore $(xs_1)s_3 = (xs_2)c_3$. In addition $(dc_1)s_3=(dc_2)c_3$ and so $((xs_1,dc_1),(xs_2,dc_2))\in\rho_X$ and the choice of $s_1, c_1$ is not important.
It is then a straightforward matter to demonstrate that this map is well-defined, the details being left to the interested reader.

It is also easy to show that if $(x,c)\rho_X\in Q_X, (s,d)\rho, (t,e)\rho\in Q$ then
$$
\left((x,c)\rho_X(s,d)\rho\right)(t,e)\rho=(x,c)\rho_X\left((s,d)\rho(t,e)\rho\right)
$$
and that $(x,c)\rho_X(1,1)\rho = (x,c)\rho_X$. It then follows that $Q$ is a semigroup and since $(1,1)\rho(s,t)\rho = (s,t)\rho$ then $Q$ is a monoid with identity $(1,1)\rho$. It is also easy to demonstrate that for all $c,d\in C,s\in S, (s,d)\rho(d,c)\rho = (s,c)\rho$, that $(c,c)\rho=(1,1)\rho$ and that $U=\{(c,d)\rho|c,d\in C\}$ is the group of units.

\smallskip

Finally the map $\iota:S\to Q$ given by $\iota(s)=(s,1)\rho$ is a monoid monomorphism. If $c\in C$ then $\iota(c)=(c,1)\rho\in U$ with inverse $(1,c)\rho$. We can then think of elements of $Q$ as being of the form $sc^{-1}$ for $s\in S, c\in C$.

The monoid $Q$ is a generalisation of the construction given in~\cite[Exercise 12.4.2, page 302]{clifford-67} and can also be found in~\cite{feller-69}. It is called a {\em (classical) monoid of right quotients of $S$ by $C$}.

\medskip

It also follows from the argument above that $Q_X$ is a right $Q-$act and hence a right $S-$act. Notice that the $S-$action is given by
$$
(x,c)\rho_X\cdot s = (x,c)\rho_X(s,1)\rho = (xs_1,c_1)\rho_X\text{ where }cs_1=sc_1.
$$
We shall call $Q_X$ the {\em act of quotients of $X$ by C}.

It is also worth noting that $(x,c)\rho_X\cdot c = (x,1)\rho_X$. Consequently if there exists $x,y\in X, s,t\in S$ such that $xs=yt$ then $(x,c)\rho_X\cdot (cs) = (y,d)\rho_X\cdot (dt)$ for any $c,d\in C$ and so we can deduce that

\begin{lemma}\label{qx-indecomposable-lemma}
Let $S$ be a CRM monoid. Then $X$ is an indecomposable $S-$act if and only if $Q_X$ is an indecomposable $Q-$act.
\end{lemma}
\begin{proof} Suppose $X$ is an indecomposable $S-$act and let $(x,c)\rho_X, (y,d)\rho_X\in Q_X$. Then there exists $x_2,\ldots, n_x\in X$, $s_1, \ldots s_n, t_2,\ldots, t_{n+1}\in S$ such that
$$
\begin{array}{rcl}
xs_1&=&x_2t_2\\
x_2s_2&=&x_3t_3\\
&\ldots&\\
x_ns_n&=&yt_{n+1}.\\
\end{array}
$$
Hence
$$
\begin{array}{rcl}
(x,c)\rho_X\cdot cs_1&=&(x_2,c)\rho_X\cdot ct_2\\
(x_2,c)\rho_X\cdot cs_2&=&(x_3,c)\rho_X\cdot ct_3\\
&\ldots&\\
(x_n,c)\rho_X\cdot cs_n&=&(y,d)\rho_X\cdot dt_{n+1}\\
\end{array}
$$
and so $Q_X$ is indecomposable.

The converse follows in a similar way.
\end{proof}
Since the cancellative elements of $Q$ are invertible then clearly every $Q-$act is weakly torsion free and so in particular $Q_X$ is weakly torsion free.
Notice that $Q_X$ is actually weakly torsion free as an $S-$act for every $S-$act $X$. 

\smallskip

Consider now the map $\theta:X\to Q_X$ given by $\theta(x)=(x,1)\rho_X$. Then $\theta(xs) = (xs,1)\rho_X = (x,1)\rho_X\cdot s=\theta(x)s$ and $\theta$ is an $S-$map. Also, if $\theta(x) = \theta(y)$ then $(x,1)\rho_X(y,1)$ and so there exists $s\in S, c\in C$ with $xs=yc$ and $1s=1c$ and so $xc=yc$. Consequently, $X$ is weakly torsion free if and only if $\theta$ is an $S-$monomorphism.

\smallskip

Suppose that in addition $X$ is also divisible. Then $X\to Q_X$ splits. To see this define $\phi:Q_X\to X$ as follows. Given $(x,c)\rho_X\in Q_X$ let $y\in X$ be the unique element such that $x=yc$ and define $\phi((x,c)\rho_X) = y$. Then it is straightforward to show that $\phi$ is a well-defined $S-$map and that $\phi\theta=1_X$. So for all $x\in X, c\in C$ we have
$$
x=\phi\theta(x) = \phi((x,1)\rho_X) = \phi((x,c)\rho_X\cdot c)=\phi((x,c)\rho_X)c.
$$

Suppose now that $\phi((x,c)\rho_X)=\phi((y,d)\rho_X)$. Then $x=\phi((x,c)\rho_X)\cdot c$ and $y=\phi((y,d)\rho_X)\cdot d$ and since $S$ satisfies the CRM property there exist $s\in S,t\in C$ such that $cs=dt$ and so $xs=\phi((x,c)\rho_X)\cdot cs=\phi((y,d)\rho_X)\cdot dt = yt$. Hence $\phi((x,c)\rho_X)=\phi((y,d)\rho_X)$ and so we deduce

\begin{lemma}\label{qxcongx-lemma}
Let $S$ be a CRM monoid and let $X$ be a weakly torsion free and divisible $S-$act. Then $Q_X\cong X$.
\end{lemma}

\smallskip
Consequently from Lemma~\ref{injective-envelope-torsion-lemma} and the well-known fact that injective acts are divisible, we get
\begin{corollary}
Let $S$ be a CRM monoid and let $X$ be a weakly torsion free $S-$act. Then $Q_{E(X)}\cong E(X)$.
\end{corollary}

\begin{proposition}
Let $S$ be a CRM monoid. The following are equivalent
\begin{enumerate}
\item for every $S-$act $X, X\cong Q_X$;
\item for every weakly torsion free right $S-$act $X, X\cong Q_X$;
\item for every torsion free right $S-$act $X, X\cong Q_X$;
\item $S\cong Q$;
\item the cancellative elements of $S$ form a group.
\end{enumerate}
\end{proposition}

\begin{proof}
(1)$\Rightarrow$(2)$\Rightarrow$(3)$\Rightarrow$(4) are obvious. (4)$\Rightarrow$(5) since the cancellative elements of $Q$ are the units. (5)$\Rightarrow$(1) since $(x,c)\rho_X=(xc^{-1},1)\rho_X$ for all $(x,c)\in Q_X$.
\end{proof}
Notice that if the cancellative elements of $S$ form a group then every $S-$act is weakly torsion free and so has a weakly torsion free cover.

\medskip

The following is fairly obvious.
\begin{lemma}
Let $S$ be a CRM monoid and let $Q$ be its monoid of quotients. Then $Q$ is a group if and only if $S$ is cancellative.
\end{lemma}

\begin{theorem}
If $S$ is a cancellative CRM monoid, then every $S-$act has a $\WT-$cover.
\end{theorem}
\begin{proof}
First, by a straightforward modification of the proof of Lemma~\ref{tf-direct-limit} we can deduce that $\WT$ is closed under directed colimits.
If $A$ is a right $S-$act, then to show that $A$ has a $\WT-$precover, it suffices by Lemma~\ref{t-precover-subact-lemma} and Lemma~\ref{torsion-free-injective-lemma} to show that $E(A)$ has an $\X-$precover where $\X=\WT\cap\D$. Let $X$ be an indecomposable, weakly torsion free and divisible $S-$act. By Lemma~\ref{qx-indecomposable-lemma} and Lemma~\ref{qxcongx-lemma} $X=Q_{X}$ is an indecomposable $Q-$act and since $Q$ is a group is therefore cyclic and so bounded in size. Hence from Proposition~\ref{lambda-skeleton-proposition} $A$ has an $\X-$precover and therefore a $\WT-$precover and so a $\WT-$cover by Theorem~\ref{precover-cover-theorem}.
\end{proof}

\smallskip
This is very similar to the approach taken by Enochs in~\cite{enochs-63} where he considers an integral domain $R$ with field of fractions $K$. It would be of interest  to determine whether there are any other CRM monoids such that the monoid of quotients $Q$ has indecomposable weakly torsion free (and divisible) acts of bounded size. That this is not always the case follows from the following example.

\medskip

\begin{example}\label{large-torsion-free-example}
\rm Let $S=\{1,0\}$ be the trivial group with a zero adjoined. Since $S$ is commutative then it is a CRM monoid. Given any set $X$, choose and fix $y \in X$ and define an $S-$action on $X$ by $x\cdot1=x$ and $x\cdot0=y$. Given any $x,x' \in X$, $x\cdot0=x'\cdot0$ and so it is easy to see that $X$ is an indecomposable $S-$act. Notice that $xc=x$ for all right cancellable elements $c \in S$ and therefore $X$ is torsion free and so weakly torsion free. It's not too hard to see that the only cyclic $S-$acts are the one element $S-$act $\Theta_S$ and $S$ itself. Therefore since $y$ is a fixed point in $X$, by \cite[Theorem III.1.8]{kilp-00}, to show $X$ is an injective $S-$act it suffices to show that any $S-$map $f:\Theta_S \to X$ extends to $S$. This is straightforward as the image of $f$ is a fixed point. We can therefore construct arbitrarily large indecomposable (weakly) torsion free injective $S-$acts over CRM monoids.
\end{example}

\section{Injective covers}

In 1981 Enochs proved that every module over a ring has an injective cover if and only if the ring is Noetherian \cite[Theorem 2.1]{enochs-81}. The situation for acts is not so straightforward. In particular if $R$ is a Noetherian ring then there exists a cardinal $\aleph$ such that every injective module is the direct sum of indecomposable injective modules of cardinality less than $\aleph$. We give an example later to show that this is not so for monoids.

It is worth noting that by \cite[Lemma III.1.7]{kilp-00} every injective $S-$act has a fixed point and that if an $S-$act $A$ has an $I-$precover then there exists $C\in\I$ such that $\hbox{\rm Hom}(C,A)\ne\emptyset$.
Recall the following result

\begin{proposition}[{\cite[Proposition III.1.13]{kilp-00}}] \label{injective-coproducts}
Let $S$ be a monoid. All coproducts of injective $S-$acts are injective if and only if $S$ is left reversible.
\end{proposition}

We have the following necessary conditions on $S$ so that all $S-$acts have an $\I-$precover.

\begin{lemma}
Let $S$ be a monoid. If every $S$-act has an $\I-$precover then
\begin{enumerate}
\item $S$ is a left reversible monoid.
\item $S$ has a left zero.
\end{enumerate}
\end{lemma}

\begin{proof}\ 

\begin{enumerate}
\item Let $A_i$, $i \in I$ be any collection of injective $S$-acts, $B=\dot\bigcup_{i \in I}A_i$ their coproduct, and $g:C \rightarrow B$ the $\I$-precover of $B$. For each $j \in I$ and inclusion $h_j:A_j \rightarrow B$ there exists an $S$-map $f_j:A_j \rightarrow C$ such that $gf_j=h_j$. Hence we can define an $S$-map $f:B \rightarrow A$ by $f|_{A_j} = f_j$ so that $gf=id_B$ and $B$ is a retract of $C$. Therefore by \cite[Proposition I.7.30]{kilp-00}, $B$ is an injective $S$-act and by Proposition \ref{injective-coproducts}, $S$ is left reversible.
\item Let $g:I \to S$ be an $\I-$precover of $S$. Since $I$ is injective it has a fixed point $z$ and so $g(z)$ is a left zero in $S$.
\end{enumerate}
\end{proof}

\begin{remark}
\rm In particular if every $S-$act has an $\I-$precover then there is a left zero $z \in S$ such that for all $s \in S$ there exists $t \in S$ with $st=z$. Obviously both conditions above are satisfied if $S$ contains a zero.

Notice also that if $S$ contains a left zero $z$ then every $S-$act contains a fixed point since if $A$ is a right $S-$act and $a\in A$ then $(az)s=az$ for all $s\in S$. Consequently $\hbox{\rm Hom}(C,A)\ne\emptyset$ for all right $S-$acts $A$ and $C$.
\end{remark}

\begin{lemma} \label{injective-coproducts-decompose}
Let $S$ be a left reversible monoid with a left zero. Then $\dot \bigcup_{i \in I}A_i \in \I$ if and only if $A_i \in \I$ for each $i \in I$.
\end{lemma}

\begin{proof}
Since $S$ is left reversible if each $A_i$ are injective then $\dot \bigcup_{i \in I}A_i$ is injective by Proposition \ref{injective-coproducts}. Conversely, assume $A=\dot \bigcup_{i \in I}A_i$ is injective, and first notice that since $S$ has a left zero each $A_i$ has a fixed point say $z_i \in A_i$. Given any $j\in I$ and monomorphism $\iota :X \to Y$ and any homomorphism $f:X \to A_j$, clearly $f \in $ Hom$(X,A)$ and so there exists $\bar{f}:Y \to A$ such that $\bar{f}|_X=f$. Now let $K_j=\{y \in Y : \bar{f}(y) \in A_j\}$ and notice that $X\subseteq K_j$ and that $y\in K_j$ if and only if $ys\in K_j$ for all $s\in S$. Now define a new function $h:Y \to A_j$ by
$$
h(y)=\begin{cases}\bar{f}(y)&y\in K_j\\z_j&\text{otherwise}\\\end{cases}
$$
Since $z_j$ is a fixed point, $h$ is a well-defined $S-$map with $h|_X=f$ and so $A_j$ is injective.
\end{proof}

\medskip

In order to apply Theorem~\ref{precover-cover-theorem} we need $\I$ to be closed under directed colimits, which in general they are not.

\bigskip

Let $S$ be a monoid and $X$ an $S-$act. We say that $X$ is {\em Noetherian} if every congruence on $X$ is finitely generated, and we say that a monoid $S$ is Noetherian if it is Noetherian as an $S-$act over itself.

\begin{lemma}[{\cite[Proposition1]{normak-77}}]\label{noetherian-proposition}
Let $S$ be a monoid and $X$ an $S-$act. Then $X$ is Noetherian if and only if $X$ satisfies the ascending chain condition on congruences on $X$.
\end{lemma}

\begin{lemma}
Let $S$ be a monoid and $X$ a Noetherian $S-$act. Then $X$ is finitely generated.
\end{lemma}
\begin{proof}
Suppose by way of contradiction that $x_1, x_2, \ldots$ is an infinite set of generators for $X$ such that for $i\ge 2$, there exists $s_i\in S$ with $x_is_i\notin x_{i-1}S$. Let $X_i = \bigcup_{j\le i}{x_jS}$ and define the right $S-$congruence $\rho_i = X_i\times X_i\cup 1_X$ and note that
$$
\rho_1\subsetneq\rho_2\subsetneq\ldots
$$
This contradicts the ascending chain condition as required.
\end{proof}

\begin{lemma}[{\cite[Proposition2, Proposition 3, Theorem 3]{normak-77}}]\label{noetherian-lemma}
Let $S$ be a monoid.
\begin{enumerate}
\item Every subact and every homomorphic image of a Noetherian $S-$act is Noetherian.
\item All finitely generated $S-$acts over a Noetherian monoid are Noetherian and finitely presented.
\end{enumerate}
\end{lemma}

It is shown in \cite[Lemma III.1.8]{kilp-00} that an $S-$act $A$ is injective if and only if for all cyclic $S-$acts $C$ and all subacts $X\subseteq C$, any $f:X \rightarrow A$ an $S$-map can be extended to $g:C\to A$ such that $g|X=f$.

\medskip

As with modules over a ring, we have

\begin{proposition} \label{injective-direct-limit}
Let $S$ be a Noetherian monoid, then every directed colimit of injective $S$-acts is injective.
\end{proposition}
\begin{proof}
Let $S$ be a Noetherian monoid, and $(A_i,\phi_{i,j})_{i \in I}$ a direct system of injective $S$-acts with directed index set $I$ and directed colimit $(A,\alpha_i)$. Since $A_i$ is injective it contains a fixed point and so $A$ contains a fixed point. Let $X \subseteq C$ be a subact of a cyclic $S-$act and $f:X \rightarrow A$ an $S$-map. Since $S$ is Noetherian, by Lemma~\ref{noetherian-lemma} $X$ is Noetherian and hence finitely generated (by ascending chain condition on Rees congruences). Therefore $f(X)=\langle a_1,\ldots,a_n \rangle$ is a finitely generated subact of $A$. Since $a_i$ are all elements of the colimit, there exists $m(1),\ldots,m(n) \in I$, and $a'_i \in A_{m(i)}$ such that $\alpha_{m(i)}(a'_i)=a_i$ for each $1 \leq i \leq n$. Since $I$ is directed, there exists some $k \in I$ with $k\ge m(1),\ldots,m(n)$ and such that $b_i=\phi_{m(i),k}(a'_i) \in A_k$. Let $B=\langle b_1,\ldots,b_n\rangle$ a finitely generated subact of $A_k$. By Lemma~\ref{noetherian-lemma}, $B$ is Noetherian and so every congruence on $B$ is finitely generated. In particular $\ker(\alpha_k|_B)=Z^\#$ is finitely generated, where $Z \subseteq B \times B$ is a finite set. So given any $(x,y) \in \ker(\alpha_k|_B)$, there exists $(p_1,q_1),\ldots,(p_m,q_m) \in Z$, $s_1,\ldots,s_m \in S$ such that $x=p_1s_1$, $q_1s_1=p_2s_2$, \ldots, $q_ms_m=y$. Now, since $\alpha_k(p_j)=\alpha_k(q_j)$, for all $1 \leq j \leq m$, by \cite[Lemma 2.1]{bailey-12}, there exists $l(j) \geq k$ such that $\phi_{k,{l(j)}}(p_j)=\phi_{k,{l(j)}}(q_j)$. Since $I$ is directed, we can take some $K \in I$ larger than all of the $l(j)$ and we have $\phi_{k,K}(p_j)=\phi_{k,K}(q_j)$ for all $1 \leq j \leq m$. Hence $\phi_{k,K}(x)=\phi_{k,K}(p_1)s_1=\phi_{k,K}(q_1)s_1=\ldots=\phi_{k,K}(q_n)s_n=\phi_{k,K}(y)$ and so $\ker(\alpha_k|_B) \subseteq \ker(\phi_{k,K})$. Hence $D=\phi_{k,K}(B)$ is a finitely generated subact of $A_K$ and $\alpha_K|_D$ is a monomorphism. Also, for $1\le i\le n$, $\alpha_K(\phi_{k,K}(b_is)) = \alpha_k(b_is) =\alpha_{m(i)}(a_i's)=a_is\in\im(f) $. Conversely given any $a_is \in \im(f)$, $a_is=\alpha_{m(i)}(a'_i)s=\alpha_K(\phi_{{m(i)},K}(a'_i))s\in\im(\alpha|_D)$ and so $\im(f)=\im(\alpha_K|_D)\cong D$. Since $A_K$ is injective, $\alpha_K^{-1}f$ can be extended to $C$ with some $S$-map $g:C \rightarrow A_K$, and so $f$ can be extended to $C$ with the $S$-map $\alpha_Kg$.
\end{proof}

\begin{theorem}\label{set-injectives--theorem}
Let $S$ be a left reversible Noetherian monoid with a left zero.  If there is a cardinal $\lambda$ such that every indecomposable injective $S-$act $X$ is such that $|X| \le \lambda$ then every $S-$act has an $\I-$cover.
\end{theorem}

\begin{proof}
By Lemma \ref{injective-coproducts-decompose} and the fact that Hom$(\Theta_S,X)\neq \emptyset$ for every $S-$act $X$ we can apply Proposition \ref{lambda-skeleton-proposition} so that every $S-$act has an $\I-$precover. By Proposition \ref{injective-direct-limit} and Theorem \ref{precover-cover-theorem} every $S-$act has an $\I-$cover.
\end{proof}

Since the monoid given in Example~\ref{large-torsion-free-example} is finite then it is clearly Noetherian. Hence it is an example of a Noetherian left reversible monoid with a left zero with arbitrarily large indecomposable injective (and torsion free) acts. Consequently, unlike in the ring case, not every monoid satisfies the conditions given in Theorem~\ref{set-injectives--theorem}.

\medskip

It is straightforward to show that the above results also hold for weakly injective acts and covers.

\section{Principally weakly injective covers}

An $S-$act is called principally weakly injective if it is injective with respect to all inclusions of principal right ideals into $S$.
Unlike injective acts principally weakly injective acts are always closed under coproducts and decompositions (\cite[Proposition III.3.4]{kilp-00}).
A straighforward modification of Theorem~\ref{set-injectives--theorem} gives us

\begin{proposition}
Let $S$ be a Noetherian monoid with a left zero. If there is a cardinal $\lambda$ such that every indecomposable principally weakly injective $S-$act $X$ is such that $|X| \le \lambda$ then every $S-$act has a $\PWI-$cover.
\end{proposition}

The purpose of the left zero in the previous lemma is to ensure that for all $S-$acts $A$ there exists an principally weakly injective act $X$ such that $\hbox{\rm Hom}_S(X,A)\ne\emptyset$.

\begin{proposition}[{\cite[Proposition III.3.2]{kilp-00}}]
Let $S$ be a monoid. Then $A$ is a principally weakly injective $S-$act if and only if for all $s \in S$ and all $S-$maps $f:sS \to A$ there exists $z \in A$ such that $f(x)=zx$ for all $x \in sS$.
\end{proposition}

\begin{lemma} \label{PWI-colimits}
If $S$ is left cancellative monoid then $\PWI$ is closed under colimits.
\end{lemma}

\begin{proof}
Let $(A_i,\phi_{i,j})_{i \in I}$ be a direct system of $S-$acts with colimit $(A,\alpha_i)$. For all $s \in S$ and $S-$maps $f:sS \to A$, let $a_i \in A_i$ for some $i \in I$ such that $\alpha_i(a_i)=f(s)$. Define $h:sS \to A_i$ by $h(st)=a_it$ for all $st \in sS$ and notice that when $S$ is left cancellative, this is a well defined $S-$map. Hence there exist $z_i\in A_i$ such that $h(x) = z_ix$ and so there exists $\alpha_i(z_i) \in A$ such that $f(x)=\alpha_ih(x) = \alpha_i(z_i)x$ for all $x \in sS$.
\end{proof}

\begin{theorem}
Let $S$ be a principally weakly self-injective left cancellative monoid. Then every $S-$act has a $\PWI-$cover with the unique mapping property.
\end{theorem}

\begin{proof}
If $S$ is principally weakly injective then $\PWI$ has a generator and so by Lemma \ref{PWI-colimits} and Theorem \ref{generator-and-colimits} every $S-$act has a $\PWI-$cover with the unique mapping property.
\end{proof}

It is clear, and rather trivial to note, that if $S$ is a monoid and $\X$ a class of $S-$acts such that every $S-$act belongs to $\X$ then every $S-$act has an $\X-$cover. Hence

\begin{theorem}
If $S$ is a regular monoid then every $S-$act has a $\PWI-$cover.
\end{theorem}

\bigskip

It is clearly of interest to determine whether every $S-$act over a left reversible Noetherian monoid with a left zero has an $\I-$cover. A knowledge of the indecomposable injective acts would help greatly in this goal.

Likewise, if $S$ is a CRM monoid with monoid of quotients $Q$ then every indecomposable, torsion free and divisible $S-$act is also an indecomposable $Q-$act. We would like to know which monoids $Q$ have indecomposable $Q-$acts of bounded size.

\end{document}